\documentclass[paper=letterpaper, onecolumn, fontsize=11pt]{article}

\usepackage[english]{babel} % spanish/English language/hyphenation
\usepackage{amsmath,amsfonts,amsthm,amssymb} % Math packages
\usepackage{latexsym,amscd,verbatim,alltt,array}
\usepackage{euscript}
\usepackage{enumerate}
\usepackage[affil-it]{authblk}
\usepackage{tikz}

\usetikzlibrary{arrows}

\usetikzlibrary{shapes, shadows, arrows}
\theoremstyle{theorem}
\newtheorem{thm}{Theorem}[section]
\newtheorem{cor}[thm]{Corollary}
\newtheorem{prop}[thm]{Proposition}
\newtheorem{lem}[thm]{Lemma}
\newtheorem{conj}[thm]{Conjecture}

\theoremstyle{definition}
\newtheorem{defn}[thm]{Definition}
\newtheorem{exmp}[thm]{Example}

\newtheorem{rem}[thm]{Remark}

\title{\vspace{-2.5cm}Monomial ideals of weighted oriented graphs}
\author{\small Yuriko Pitones\thanks{\texttt{ypitones@math.cinvestav.edu.mx}}\ }
\author{\small Enrique Reyes\thanks{\texttt{ereyes@math.cinvestav.mx}}\ }
\author{\small Jonathan Toledo\thanks{\texttt{jtt@math.cinvestav.mx}}}
\affil{\vspace{-0.2cm}\scriptsize Departamento de Matem\'{a}ticas\\ Centro de Investigaci\'{o}n y de Estudios Avanzados del Instituto Polit\'ecnico Nacional\\ Apartado Postal 14--740, Ciudad de M\'{e}xico\\ 07000 M\'exico}

\date{\tiny\ }

\begin{document}

\maketitle
\thispagestyle{empty}

\vspace{-10ex}

\parindent=8mm

\begin{abstract}
\noindent
%[Here your abstract]
Let $I=I(D)$ be the edge ideal of a weighted oriented graph $D$. We determine the irredundant irreducible decomposition of $I$. Also, we characterize the associated primes and the unmixed property of $I$. Furthermore, we give a combinatorial characterization for the unmixed property of $I$, when $D$ is bipartite, $D$ is a whisker or $D$ is a cycle. Finally, we study the Cohen-Macaulay property of $I$.%for some families of weighted oriented graphs.
\end{abstract}

\noindent
{\small\textbf{Keywords:} Weighted oriented graphs, unmixed property, irreducible decomposition, Cohen-Macaulay property.}

\section{Introduction}

A {\it weighted oriented graph\/} is a triplet $D=(V(D),E(D),w)$, where $V(D)$ is a finite set, $E(D) \subseteq V(D) \times V(D)$ and $w$ is a function $w:V(D) \to\mathbb{N}$. The vertex set of $D$ and the edge set of $D$ are $V(D)$ and $E(D)$, respectively. Some times for short we denote these sets by $V$ and $E$ respectively. The \emph{weight} of $x\in V$ is $w(x)$. If $e=(x,y)\in E$, then $x$ is the tail of $e$ and $y$ is the head of $e$. The {\it underlying graph\/} of $D$ is the simple graph $G$ whose vertex set is $V$ and whose edge set is $\{\{x,y\}| (x,y)\in E\}$. If $V(D)=\{x_1,\ldots,x_n\}$, then we consider the polynomial ring $R=K[x_1,\ldots,x_n]$ in $n$ variables over a field $K$. In this paper we introduce and study the edge ideal of $D$ given by $I(D)=(x_{i}x_{j}^{w(x_{j})}:(x_{i},x_{j})\in E(D))$ in $R$, (see Definition \ref{I(D)}). 

 \vspace{1ex}

In Section 2 we study the vertex covers of $D$. In particular we introduce the notion of strong vertex cover (Definition \ref{strong-cover-defn}) and we prove that a minimal vertex cover is strong. In Section 3 we characterize the irredundant irreducible decomposition of $I(D)$. In particular we show that the minimal monomial irreducible ideals of $I(D)$ are associated with the strong vertex covers of $D$. In Section 4 we give the following characterization of the unmixed property of $I(D)$. 

 \vspace{1ex}

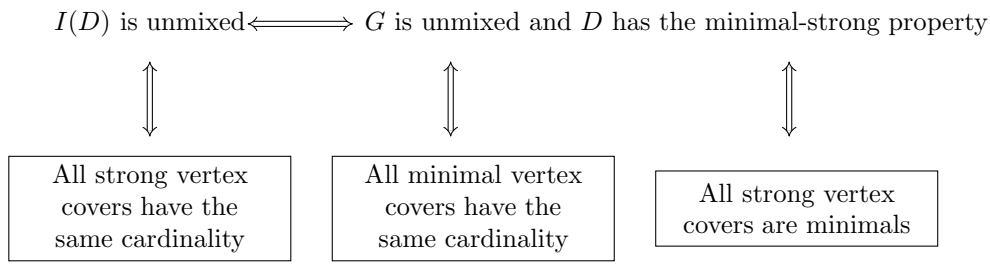
\tikzstyle{block} = [draw, rectangle, text width=10em, text centered, minimum height=10mm, node distance=5em]
\begin{center}
\noindent
\begin{tikzpicture}

\draw[implies-implies,double equal sign distance] (1.3,0) -- (2.7,0);%arriba
\draw[implies-implies,double equal sign distance] (0,-.5) -- (0,-1.5);%primera
\draw[implies-implies,double equal sign distance] (4,-.5) -- (4,-1.5);%segunda
\draw[implies-implies,double equal sign distance] (8.5,-.5) -- (8.5,-1.5);%tercera

\node (D-U) at (0,0) {$I(D)$ is unmixed};
\node (G-U) at (7,0) {$G$ is unmixed and $D$ has the minimal-strong property};
\node[block, yshift=-5em](S-M) at (0,-.7) {All strong vertex covers have the same cardinality};
\node[block, yshift=-5em](M-C) at (4.3,-.7) {All minimal vertex covers have the same cardinality};
\node[block, yshift=-5em](S-C) at (8.6,-.7) {All strong vertex covers are minimals};

\end{tikzpicture}
\end{center}

 \vspace{1ex}
 
Furthermore, if $D$ is bipartite, $D$ is a whisker or $D$ is a cycle, we give an effective (combinatorial) characterization of the unmixed property. Finally in Section 5 we study the Cohen-Macaulayness  of $I(D)$. In particular we characterize the Cohen-Macaulayness when $D$ is a path or $D$ is complete. Also, we give an example where this property depend of the characteristic of the field $K$. 

\section{Weighted oriented graphs and their vertex covers}

In this section we define the weighted oriented graphs and we study their vertex covers. Furthermore, we define the strong vertex covers and we characterize when $V(D)$ is a strong vertex cover of $D$. In this paper we denote the set $\{x\in V\mid w(x)\neq 1\}$ by $V^{+}$. 

\begin{defn}
A {\it vertex cover\/} $C$ of $D$ is a subset of $V$, such that if $(x,y)\in E$, then $x\in C$ or $y\in C$. A vertex cover $C$ of $D$ is {\it minimal\/} if each proper subset of $C$ is not a vertex cover of $D$.
\end{defn}

\begin{defn}
Let $x$ be a vertex of a weighted oriented graph $D$, the sets $N_{D}^{+}(x)=\{y:(x,y)\in E(D)\}$ and $N_{D}^{-}(x)=\{y:(y,x)\in E(D)\}$ are called the {\it out-neighbourhood\/} and the {\it in-neighbourhood\/} of $x$, respectively. Furthermore, the {\it neighbourhood\/} of $x$ is the set $N_{D}(x)=N_{D}^{+}(x)\cup N_{D}^{-}(x)$. 
\end{defn}

\begin{defn}\label{L-sets}
Let $C$ be a vertex cover of a weighted oriented graph $D$, we define
\[
L_1(C)=\{x\in C\mid N_{D}^{+}(x)\cap C^{c}\neq \emptyset \},
\]
\[
L_2(C)=\{x\in C\mid\mbox{$x\notin L_1(C)$ and $N^{-}_{D}(x)\cap C^c\neq\emptyset$}\} \text{ and } 
\]
\[
L_3(C)=C\setminus(L_1(C)\cup L_2(C)),
\]
where $C^{c}$ is the complement of $C$, i.e. $C^{c}=V\setminus C$.
\end{defn}

\begin{prop}\label{remark-L3}
If $C$ is a vertex cover of $D$, then 
\begin{center}
$L_{3}(C)=\{x\in C \mid N_{D}(x)\subset C\}$.
\end{center}
\end{prop}

\proof
If $x\in L_{3}(C)$, then $N_{D}^{+}(x)\subseteq C$, since $x\notin L_{1}(C)$. Furthermore $N^{-}_{D}(x)\subseteq C$, since $x\notin L_{2}(C)$. Hence $N_{D}(x)\subset C$, since $x\notin N_{D}(x)$. Now, if $x\in C$ and $N_{D}(x)\subset C$, then $x\notin L_{1}(C)\cup L_{2}(C)$. Therefore $x\in L_{3}(C)$. %Therefore $L_{3}(C)=\{x\in C\mid N_{D}(x)\subset C\}$.
\qed

\begin{prop}\label{minimalcover-L3}
If $C$ is a vertex cover of $D$, then $L_{3}(C)=\emptyset$ if and only if $C$ is a minimal vertex cover of $D$.
\end{prop}

\proof
 $\Rightarrow)$ If $x\in C$, then by Proposition \ref{remark-L3} we have $N_{D}(x)\not\subset C$, since $L_{3}(C)=\emptyset$. Thus, there is $y\in N_{D}(x)\setminus C$ implying $C\setminus \{x\}$ is not a vertex cover. Therefore, $C$ is a minimal vertex cover. 
  
\vspace{1ex}
\noindent
$\Leftarrow)$ If $x\in L_{3}(C)$, then by Proposition \ref{remark-L3}, $N_{D}(x)\subseteq C\setminus \{x\}$. Hence, $C\setminus\{x\}$ is a vertex cover. A contradiction, since $C$ is minimal. Therefore $L_{3}(C)=\emptyset$. 

\begin{defn}\label{strong-cover-defn}
A vertex cover $C$ of $D$ is {\it strong\/} if for each $x\in L_3(C)$ there is $(y,x)\in E(D)$ such that $y\in L_2(C)\cup L_{3}(C)$ with $y\in V^{+}$ (i.e. $w(y)\neq 1$).
\end{defn}

\begin{rem}\label{strong-cover}
Let $C$ be a vertex cover of $D$. Hence, by Proposition \ref{remark-L3} and since $C=L_{1}(C)\cup L_{2}(C)\cup L_{3}(C)$, we have that $C$ is strong if and only if for each $x\in C$ such that $N(x)\subset C$, there exist $y\in N^{-}(x)\cap(C\setminus L_{1}(C))$ with $y\in V^{+}$.
\end{rem}

\begin{cor}\label{C-minimal-strong}
If $C$ is a minimal vertex cover of $D$, then $C$ is strong.
\end{cor}
\proof
By Proposition \ref{minimalcover-L3}, we have $L_{3}(C)=\emptyset$, since $C$. Hence, $C$ is strong.
\qed

\begin{rem}\label{V-strong}
The vertex set $V$ of $D$ is a vertex cover. Also, if $z\in V$, then $N_{D}(z)\subseteq V\setminus z$. Hence, by Proposition \ref{remark-L3}, $L_{3}(V)=V$. Consequently, $L_{1}(V)=L_{2}(V)=\emptyset$. By Proposition \ref{minimalcover-L3}, $V$ is not a minimal vertex cover of $D$. Furthermore since $L_{3}(V)=V$, $V$ is a strong vertex cover if and only if $N_{D}^{-}(x)\cap V^{+}\neq \emptyset$ for each $x\in V$.
\end{rem}	

\begin{defn}
If $G$ is a cycle with $E(D) = \{(x_1,x_2), \ldots, (x_{n-1},x_n), (x_n,x_1)\}$ and $V(D)=\{x_1,\ldots, x_n\}$, then $D$ is called {\it oriented cycle\/}. 
\end{defn}

\begin{defn}
$D$ is called {\it unicycle oriented graph\/} if it satisfies the following conditions:
\begin{enumerate}[1)]
\item The underlying graph of $D$ is connected and it has exactly one cycle $C$.
\item $C$ is an oriented cycle in $D$. Furthermore for each $y\in V(D)\setminus V(C)$, there is an oriented path from $C$ to $y$ in $D$.
\item $w(x)\neq 1$ if $deg_{G}(x)\geq 1$. 
\end{enumerate}
\end{defn}

\begin{lem}\label{lemma2-V} 
If $V(D)$ is a strong vertex cover of $D$ and $D_{1}$ is a maximal unicycle oriented subgraph of $D$, then $V(D^\prime)$ is a strong vertex cover of $D^\prime=D\setminus V(D_{1})$.
\end{lem}

\proof
We take $x\in V(D^\prime)$. Thus, by Remark \ref{V-strong},  there is $y\in N_{D}^{-}(x)\cap V^{+}(D)$. If $y\in D_{1}$, then we take $D_{2}=D_{1}\cup \{(y,x)\}$. Hence, if $C$ is the oriented cycle of $D_{1}$, then $C$ is the unique cycle of $D_{2}$, since $deg_{D_{2}}(x)=1$. If $y\in C$, then $(y,x)$ is an oriented path from $C$ to $x$. Now, if $y\notin C$, then there is an oriented path $\mathcal{L}$ form $C$ to $y$ in $D_1$. Consequently, $\mathcal{L } \cup \{(y,x)\}$ is an oriented path form $C$ to $x$. Furthermore, $deg_{D_{2}}(x)=1$ and $w(y)\neq 1$, then $D_{2}$ is an unicycle oriented graph. A contradiction, since $D_{1}$ is maximal. This implies $y\in V(D^\prime)$, so $y\in N^{-}_{D^\prime}(x)\cap V^{+}(D^{\prime})$. Therefore, by Remark \ref{V-strong}, $V(D^\prime)$ is a strong vertex cover of $D^\prime$. 
\qed

\begin{lem}\label{lemma1-V}
If $V(D)$ is a strong vertex cover of $D$, then there is an unicycle oriented subgraph of $D$. 
\end{lem}

\proof
Let $y_{1}$ be a vertex of $D$. Since $V=V(D)$ is a strong vertex cover, there is $y_2\in V$ such that $y_2\in N^{-}(y_1)\cap V^{+}$. Similarly, there is $y_3\in N^{-}(y_2)\cap V^{+}$. Consequently, $(y_3,y_2,y_1)$ is an oriented path. Continuing this process, we can assume there exist $y_2,y_3,\ldots, y_{k}\in V^{+}$ where $(y_k,y_{k-1},\ldots,y_2,y_1)$ is an oriented path and there is $1\leq j\leq k-2$ such that $(y_{j},y_k)\in E(D)$, since $V$ is finite. Hence, $C=(y_k,y_{k-1},\ldots, y_j,y_k)$ is an oriented cycle and $\mathcal{L}=(y_j,\ldots,y_1)$ is an oriented path form $C$ to $y_1$. Furthermore, if $j=1$, then $w(y_1)\neq 1$. Therefore, $D_1=C\cup \mathcal{L}$ is an unicycle oriented subgraph of $D$. 
\qed

\begin{prop}\label{prop-unicycle}
Let  $D=(V,E,w)$ be a weighted oriented graph, hence $V$ is a strong vertex cover of $D$ if and only if there are $D_{1},\ldots, D_s$ unicycle oriented subgraphs of $D$ such that $V(D_1),\ldots, V(D_s)$ is a partition of $V=V(D)$.
\end{prop}

\proof
$\Rightarrow)$ By Lemma \ref{lemma1-V}, there is a maximal unicycle oriented subgraph $D_1$ of $D$. Hence, by Lemma \ref{lemma2-V}, $V(D^\prime)$ is a strong vertex cover of $D^\prime=D\setminus V(D_1)$. So, by Lemma \ref{lemma1-V}, there is $D_2$ a maximal unicycle oriented subgraph of $D^\prime $. Continuing this process we obtain unicycle oriented subgraphs $D_1,\ldots, D_s$ such that $V(D_1),\ldots, V(D_s)$ is a partition of $V(D)$.   
  
\vspace{1ex}
\noindent
$\Leftarrow)$ We take $x\in V(D)$. By hypothesis there is $1\leq j\leq s$ such that $x\in V(D_j)$. We assume $C$ is the oriented cycle of $D_j$. If $x\in V(C)$, then there is $y\in V(C)$ such that $(y,x)\in E(D_{j})$ and $w(y)\neq 1$, since $deg_{D_j}(y)\geq 2$ and $D_{j}$ is a unicycle oriented subgraph. Now, we assume $x\notin V(C)$, then there is an oriented path $\mathcal{L}=(z_1,\ldots, z_r)$ such that $z_1\in V(C)$ and $z_r=x$. Thus, $(z_{r-1},x)\in E(D)$. Furthermore, $w(z_{r-1})\neq 1$, since $deg_{D_{j}}(z_{r-1})\geq 2$. Therefore $V$ is a strong vertex cover.
\qed

\section{Edge ideals and their primary decomposition}

As is usual if $I$ is a monomial ideal of a polynomial ring $R$, we denote by $\mathcal{G}(I)$ the minimal monomial set of generators of $I$.  Furthermore, there exists a unique decomposition, $I=\mathfrak{q}_{1}\cap \cdots \cap \mathfrak{q}_{r}$, where $\mathfrak{q}_{1},\ldots, \mathfrak{q}_{r}$ are irreducible monomial ideals such that $I\neq \bigcap_{i\neq j}^{}\mathfrak{q}_{i}$ for each $j=1,\ldots, r$. This is called the irredundant irreducible decomposition of $I$. Furthermore, $\mathfrak{q}_{i}$ is an irreducible monomial ideal if and only if $\mathfrak{q}_{i}=(x_{i_{1}}^{a_{1}},\ldots , x_{i_{s}}^{a_{s}})$ for some variables $x_{i_{j}}$. Irreducible ideals are primary, then a irreducible decomposition is a primary decomposition. For more details of primary decomposition of monomial ideals see \cite[Chapter 6]{monalg-2}.  In this section, we define the edge ideal $I(D)$ of a weighted oriented graph $D$ and we characterize its irredundant irreducible decomposition. In particular we prove that this decomposition is an irreducible primary decomposition, i.e, the radicals of the elements of the irredundant irreducible decomposition of $I(D)$ are different. %then $rad(\mathfrak{q}_{i})\neq rad(\mathfrak{q}_{j})$ for $1\leq i<j\leq r$.

\begin{defn}\label{I(D)}
Let $D=(V,E,w)$ be a weighted oriented graph with $V=\{x_{1},\ldots, x_{n}\}$. The {\it edge ideal of $D$\/}, denote by $I(D)$, is the ideal of $R=K[x_{1},\ldots, x_{n}]$ generated by $\{x_{i}x_{j}^{w(x_{j})} \mid (x_{i},x_{j})\in E\}$.
\end{defn}

\begin{defn}
A {\it source\/} of $D$ is a vertex $x$, such that $N_{D}(x)=N_{D}^{+}(x)$. A {\it sink\/} of $D$ is a vertex $y$ such that $N_{D}(y)=N_{D}^{-}(y)$. 
\end{defn}

\begin{rem}\label{rem3.3}
Let $D=(V,E,w)$ be a weighted oriented graph. We take $D^\prime=(V,E,w^\prime)$ a weighted oriented graph such that $w^\prime(x)=w(x)$ if $x$ is not a source and $w^\prime(x)=1$ if $x$ is a source. Hence, $I(D)=I(D^\prime)$. For this reason in this paper, we will always assume that if $x$ is a source, then $w(x_{i})=1$.

\end{rem}

\begin{defn}
Let $C$ be a vertex cover of $D$,
 the {\it irreducible ideal associated to $C$\/} is the ideal
\[
I_C=\bigl( L_1(C)\cup\{x_j^{w(x_{j})}\mid x_j\in L_2(C)\cup L_3(C)\}\bigr).
\]
\end{defn}

\begin{lem}\label{rem01}
$I(D)\subseteq I_C$ for each vertex cover $C$ of $D$.
\end{lem}
\proof
We take $I=I(D)$ and $m\in \mathcal{G}(I)$, then $m=xy^{w(y)}$, where $(x,y)\in D$. Since $C$ is a vertex cover, $x\in C$ or $y\in C$. If $y\in C$, then $y \in I_{C}$ or $y^{w(y)}\in I_{C}$. Thus, $m=xy^{w(y)}\in I_{C}$. Now, we assume $y\notin C$, then $x\in C$. Hence, $y\in N_{D}^{+}(x)\cap C^{c}$, so $x\in L_{1}(C)$. Consequently, $x\in I_{C}$ implying $m=xy^{w(y)}\in I_{C}$. Therefore $I\subseteq I_{C}$.  
\qed

\begin{defn}
Let $I$ be a monomial ideal. An irreducible monomial ideal $\mathfrak{q}$ that contains $I$ is called a {\it minimal irreducible monomial ideal of $I$\/} if for any irreducible monomial ideal $\mathfrak{p}$ such that $I\subseteq \mathfrak{p}\subseteq \mathfrak{q}$ one has that $\mathfrak{p}=\mathfrak{q}$.
\end{defn}

\begin{lem}\label{lemma1}
Let $D$ be a weighted oriented graph. If $I(D)\subseteq (x_{i_{1}}^{a_{1}},\ldots , x_{i_{s}}^{a_{s}})$, then $\{x_{i_{1}},\ldots, x_{i_{s}}\}$ is a vertex cover of $D$.
\end{lem}
\proof
We take $J=(x_{i_{1}}^{a_{1}},\ldots, x_{i_{s}}^{a_{s}})$. If $(a,b)\in E(D)$, then $ab^{w(b)}\in I(D)\subseteq J$. Thus, $x_{i_{j}}^{a_{j}}|ab^{w(b)}$ for some $1\leq j\leq s$. Hence, $x_{i_{j}}\in \{a,b\}$ and $\{a,b\}\cap\{x_{i_{1}}\ldots x_{i_{s}}\}\neq \emptyset$. Therefore $\{x_{i_{1}},\ldots, x_{i_{s}}\}$ is a vertex cover of $D$.
\qed

\begin{lem}\label{lemma2}
Let $J$ be a minimal irreducible monomial ideal of $I(D)$ where $\mathcal{G}(J)=\{x_{i_{1}}^{a_{s}}, \ldots, x_{i_{s}}^{a_{s}}\}$. If $a_{j}\neq 1$ for some $1\leq j \leq s$, then there is $(x,x_{i_{j}})\in E(D)$ where $x\notin \mathcal{G}(J)$.
\end{lem}
\proof
By contradiction suppose there is $a_{j}\neq 1$ such that if $(x,x_{i_{j}})\in E(D)$, then $x\in M=\{x_{i_{1}}^{a_{1}},\ldots,x_{i_{s}}^{a_{s}}\}$. We take the ideal $J^\prime=(M\setminus \{x_{i_{j}}^{a_{j}}\})$. If $(a,b)\in E(D)$, then $ab^{w(b)}\in I(D)\subseteq J$. Consequently, $x_{i_{k}}^{a_{k}}|ab^{w(b)}$ for some $1\leq k\leq s$. If $k\neq j$, then $ab^{w(b)}\in J^\prime$. Now, if $k=j$, then by hypothesis $a_{j}\neq 1$. Hence, $x_{i_{j}}^{a_{j}}|b^{w(b)}$ implying $x_{i_{j}}=b$. Thus, $(a,x_{i_{j}})\in E(D)$, so by hypothesis $a\in M\setminus \{x_{i_{j}}^{a_{j}}\}$. This implies $ab^{w(b)}\in J^\prime$. Therefore $I(D)\subseteq J^{\prime}\subsetneq J$. A contradiction, since $J$ is minimal. 

\qed

\begin{lem}\label{lemma3}
Let $J$ be a minimal irreducible monomial ideal of $I(D)$ where $\mathcal{G}(J)=\{x_{i_{1}}^{a_{1}},\ldots, x_{i_{s}}^{a_{s}}\}$. If $a_{j}\neq 1$ for some $1\leq j\leq s$, then $a_{j}=w(x_{i_{j}})$.
\end{lem}
\proof
By Lemma \ref{lemma2}, there is $(x,x_{i_{j}})\in E(D)$ with $x\notin M=\{x_{i_{1}}^{a_{1}},\ldots, x_{i_{s}}^{a_{s}}\}$. Also, $xx_{i_{j}}^{w(x_{i_{j}})}\in I(D)\subseteq J$, so $x_{i_{k}}^{a_{k}}|xx_{i_{j}}^{w(x_{i_{j}})}$ for some $1\leq k\leq s$. Hence, $x_{i_{k}}^{a_{k}}|x_{i_{j}}^{w(x_{i_{j}})}$, since $x\notin M$. This implies, $k=j$ and $a_{j}\leq w(x_{i_{j}})$. If $a_{j}<w(x_{i_{j}})$, then we take $J^{\prime}=(M^{\prime})$ where $M^\prime=\{M\setminus\{x_{i_{j}}^{a_{j}}\}\}\cup \{x_{i_{j}}^{w(x_{i_{j}})}\}$. So, $J^{\prime}\subsetneq J$. Furthermore, if $(a,b)\in E(D)$, then $m=ab^{w(b)}\in I(D)\subseteq J$. Thus, $x_{i_{k}}^{a_{k}}| ab^{w(b)}$ for some $1\leq k\leq s$. If $k\neq j$, then $x_{i_{k}}^{a_{k}}\in M^{\prime}$ implying $ab^{w(b)}\in J^{\prime}$. Now, if $k=j$ then $x_{i_{j}}^{a_{j}}|b^{w(b)}$, since $a_{j}>1$. Consequently, $x_{i_{j}}=b$ and $x_{i_{j}}^{w(x_{i_{j}})}|m$. Then $m\in J^\prime$. Hence $I(D)\subseteq J^\prime\subsetneq J$, a contradiction since $J$ is minimal. Therefore $a_{j}=w(x_{i_{j}})$. 
\qed

\begin{thm}\label{thm-irredDecom}
The following conditions are equivalent:
\begin{enumerate}[1)]
\item $J$ is a minimal irreducible monomial ideal of $I(D)$.
\item There is a strong vertex cover $C$ of $D$ such that $J=I_C$. 
\end{enumerate}
\end{thm}

\proof
{\it 2)} $\Rightarrow$ {\it 1)} By definition $J=I_C$ is a monomial irreducible ideal. By Lemma~\ref{rem01}, $I(D)\subseteq J$. Now, suppose $I(D)\subseteq J^\prime\subseteq J$, where $J^\prime$ is a monomial irreducible ideal. We can assume $\mathcal{G}(J^\prime)=\{x_{j_{1}}^{b_{1}},\ldots, x_{j_{s}}^{b_{s}}\}$. If $x\in L_1(C)$, then there is $(x,y)\in E(D)$ with $y\notin C$. Hence, $xy^{w(y)}\in I(D)$ and $y^r\notin J$ for each $r\in\mathbb{N}$. Consequently $y^r\notin J^\prime$ for each $r$, implyinig $y\notin\{x_{j_{1}},\ldots, x_{j_{s}}\}$. Furthermore $x_{j_{i}}^{b_{i}}|xy^{w(y)}$ for some $1\leq  i \leq s$, since $xy^{w(y)}\in I(D)\subseteq J^\prime$. This implies, $x=x_{j_{i}}^{b_{i}}\in J^\prime$. Now, if $x\in L_2(C)$, then there is $(y,x)\in E(D)$ with $y\notin C$. Thus $y\notin J$, so $y\notin \{x_{j_{1}}^{b_{1}},\ldots, x_{j_{s}}^{b_{s}}\}$. Also, $x^{w(x)}y\in I(D)\subseteq J^\prime$, then $x_{j_{i}}^{b_{i}}|x^{w(x)}y$ for some $1\leq i\leq s$. Consequently, $x_{j_{i}}^{b_{i}}|x^{w(x)}$ implies $x^{w(x)}\in J^{\prime}$. Finally if $x\in L_3(C)$, then there is $(y,x)\in E(D)$ where $y\in L_2(C)\cup L_{3}(C)$ and $w(y)\neq 1$, since $C$ is a strong vertex cover. So, $x^{w(x)}y\in I(D)\subseteq J^\prime$ implies $x_{j_{i}}^{b_{i}}|x^{w(x)}y$ for some $i$. Furthermore $y\notin J=I_{C}$, since $y\in L_{2}(C)\cup L_{3}(C)$ and $w(y)\neq 1$. This implies $y\notin J^\prime$ so, $x_{j_{i}}^{b_{i}}|x^{w(x)}$ then $x^{w(x)}\in J^\prime$. Hence, $J=I_{C}\subseteq J^\prime$. Therefore, $J$ is a minimal monomial irreducible of $I(D)$.

\vspace{1ex}

\noindent
{\it 1)} $\Rightarrow$ {\it 2)} Since $J$ is irreducible, we can suppose $\mathcal{G}(J)=\{x_{i_{1}}^{a_{1}},\ldots, x_{i_{s}}^{a_{s}}\}$. By Lemma \ref{lemma3}, we have $a_{j}=1$ or $a_{j}=w(x_{i_{j}})$ for each $1\leq j\leq s$. Also, by Lemma \ref{lemma1}, $C=\{x_{i_{1}},\ldots, x_{i_{s}}\}$ is a vertex cover of $D$. We can assume $\mathcal{G}(I_{C})=\{x_{i_{1}}^{b_{1}},\ldots, x_{i_{s}}^{b_{s}}\}$, then $b_{j}\in \{1,w(x_{i_{j}})\}$ for each $1\leq j\leq s$. Now, suppose $b_{k}=1$ and $w(x_{i_{k}})\neq 1$ for some $1\leq k\leq s$. Consequently $x_{i_{k}}\in L_{1}(C)$. Thus, there is $(x_{i_{k}},y)\in E(D)$ where $y\notin C$. So, $x_{i_{k}}y^{w(y)}\in I(D)\subseteq J$ and $x_{i_{r}}^{a_{r}}|x_{i_{k}}y^{w(y)}$ for some $1\leq r\leq s$. Furthermore $y\notin C$, then $r=k$ and $a_{k}=a_{r}=1$. Hence, $I_{C}\cap V(D)\subseteq J\cap V(D)$. This implies, $I_{C}\subseteq J$, since $a_{j},b_{j}\in \{1,w(x_{i_{j}})\}$ for each $1\leq j\leq s$. Therefore $J=I_{C}$, since $J$ is minimal. In particular $a_{i}=b_{i}$ for each $1\leq i\leq s$.

\vspace{1ex}
\noindent
Now, assume $C$ is not strong, then there is $x\in L_{3}(C)$ such that if $(y,x)\in E(D)$, then $w(y)=1$ or $y\in L_{1}(C)$. We can assume $x=x_{i_{1}}$, and we take $J^\prime$ the monomial ideal with $\mathcal{G}(J^\prime)=\{x_{i_{2}}^{a_{2}},\ldots, x_{i_{s}}^{a_{s}}\}$. We take $(z_{1},z_{2})\in E(D)$. If $x_{i_{j}}^{a_{j}}| z_{1}z_{2}^{w(z_{2})}$ for some $2\leq j\leq s$, then $z_{1}z_{2}^{w(z_{2})}\in J^\prime$. Now, assume $x_{i_{j}}^{a_{j}}\nmid z_{1}z_{2}^{w(z_{2})}$ for each $2\leq j\leq s$. Consequently $z_{2}\notin\{x_{i_{2}}\ldots x_{i_{s}}\}$, since $a_{j}\in\{1,w(x_{i_{j}})\}$. Also $z_{1}z_{2}^{w(z_{2})}\in I(D)\subseteq J$, then $x_{i_{1}}^{a_{1}}| z_{1}z_{2}^{w(z_{2})}$. But $x_{i_{1}}\in L_{3}(C)$, so $z_{1}, z_{2}\in N_{G}[x_{i_{1}}]\subseteq C$. If $x_{i_{1}}=z_{1}$, then there is $2\leq r\leq s$ such that $z_{2}=x_{i_{r}}$. Thus $x_{i_{r}}^{a_{r}}\mid z_{1}z_{2}^{w(z_{2})}$. A contradiction, then $x_{i_{1}}=z_{2}$, $z_{1}\in C$ and $(z_{1},x_{i_{1}})\in E(D)$. Then, $w(z_{1})=1$ or $z_{1}\in L_{1}(C)$. In both cases $z_{1}\in \mathcal{G}(I_{C})$. Furthermore $z_{1}\neq z_{2}$ since $(z_{1}, z_{2})\in E(D)$. This implies $z_{1}\in \mathcal{G}(J^{\prime})$. So, $z_{1}z_{2}^{w(z_{2})}\in J^{\prime}$. Hence, $I(D)\subseteq J^{\prime}$. This is a contradiction, since $J$ is minimal. Therefore $C$ is strong. 

\qed

\begin{thm}\label{decomposition}
If $\mathcal{C}_{s}$ is the set of strong vertex covers of $D$, then the irredundant irreducible decomposition of $I(D)$ is given by $I(D)=\bigcap_{C\in \mathcal{C}_{s}} I_{C}$. 
\end{thm}

\proof
By \cite[Theorem 1.3.1]{Herzog-Hibi}, there is a unique irredundant irreducible decomposition $I(D)=\bigcap_{i=1}^{m}I_{i}$. If there is an irreducible ideal $I_{j}^{\prime}$ such that $I(D)\subseteq I_{j}^{\prime}\subseteq I_{j}$ for some $j\in\{1,\ldots,m\}$, then $I(D)=(\bigcap_{i\neq j}^{}I_{i})\cap I_{j}^{\prime}$ is an irreducible decomposition. Furthermore this decomposition is irredundant. Thus, $I_{j}^{\prime}=I_{j}$. Hence, $I_{1},\ldots, I_{m}$ are minimal irreducible ideals of $I(D)$. Now, if there is $C\in \mathcal{C}_{s}$ such that $I_{C}\notin\{I_{1},\ldots,I_{m}\}$, then there is $x_{j_{i}}^{\alpha_{i}}\in I_{i}\setminus I_{C}$ for each $i\in\{1,\ldots,m\}$. Consequently, $m=lcm(x_{j_{1}}^{\alpha_{1}},\ldots, x_{j_{m}}^{\alpha_{m}})\in\bigcap_{i=1}^{m}I_{i}=I(D)\subseteq I_{C}$. Furthermore, if $C=\{x_{i_1},\ldots, x_{i_{k}}\}$, then $I_{C}=(x_{i_1}^{\beta_1},\ldots, x_{i_{k}}^{\beta_k})$ where $\beta_{j}\in\{1,w(x_{i_{j}})\}$. Hence, there is $j\in \{1,\ldots,k\}$ such that $x_{i_{j}}^{\beta_j}|m$. So, there is $1\leq u\leq m$ such that $x_{i_{j}}^{\beta_j}\mid x_{j_{u}}^{\alpha_u}$. A contradiction, since $x_{j_{u}}^{\alpha_{u}}\notin I_{C}$. Therefore $I(D)=\bigcap_{C\in \mathcal{C}_{s}}I_{C}$ is the irredundant irreducible decomposition of $I(D)$. 
\qed

\begin{rem}\label{ass-primes}
If $C_{1},\ldots, C_{s}$ are the strong vertex covers of $D$, then by Theorem \ref{decomposition}, $I_{C_{1}}\cap\cdots\cap I_{C_{s}}$ is the irredundant irreducible decomposition of $I(D)$. Furthermore, if $P_i=rad(I_{C_i})$, then $P_{i}=(C_i)$. So, $P_{i}\neq P_{j}$ for $1\leq i<j\leq s$. Thus, $I_{C_{1}}\cap\cdots\cap I_{C_{s}}$ is an irredundant primary decomposition of $I(D)$. In particular we have ${\rm Ass}(I(D))=\{P_{1},\ldots, P_{s}\}$.
\end{rem}

\begin{exmp}\label{example1}
Let $D$ be the following oriented weighted graph 
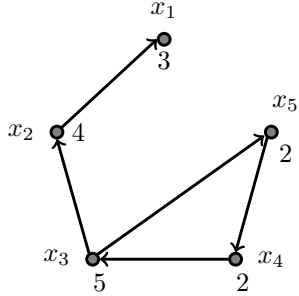
\begin{center}
\begin{tikzpicture}[line width=1.1pt,scale=0.95]
		\tikzstyle{every node}=[inner sep=0pt, minimum width=4.5pt] 

\draw [->] (-1.45,.45)--(-0.1,1.7); %x2-x1
\draw [->] (-1.05,-1.3)--(-1.5,.3); %x3-x2
\draw [->] (-.93,-1.31) --(1.4,.35); %x3-x5
\draw [->] (.9,-1.38)--(-.9,-1.38); %x4-x3
\draw [->] (1.45,.35)--(1,-1.28); %x5-x4

\draw (-1,-1.38) node (v3) [draw, circle, fill=gray] {};
\draw (1,-1.38) node (v4) [draw, circle, fill=gray] {};
\draw (1.5,.4) node (v5) [draw, circle, fill=gray] {};
\draw (0,1.7) node (v1)[draw, circle, fill=gray] {};
\draw (-1.5,.4) node (v2) [draw, circle, fill=gray] {};

\node at (-1.5,-1.35) {$x_{3}$};
\node at (1.5,-1.38) {$x_{4}$};
\node at (1.7,.8) {$x_{5}$};
\node at (0,2.1) {$x_{1}$};
\node at (-2,.4) {$x_{2}$};

\node at (-.9,-1.7) {$5$};
\node at (1.1,-1.7) {$2$};
\node at (1.7,.10) {$2$};
\node at (0,1.4) {$3$};
\node at (-1.2,.4) {$4$};

\end{tikzpicture}
\end{center}
\noindent
whose edge ideal is $I(D)=(x_{1}^{3}x_{2},x_{2}^{4}x_{3},x_{3}^{5}x_{4},x_{3}x_{5}^{2},x_{4}^{2}x_{5})$. From Theorem \ref{thm-irredDecom} and Theorem \ref{decomposition}, the irreducible decomposition of $I(D)$ is:  
\begin{center}
$I(D)=(x_{1}^{3},x_{3},x_{4}^{2})\cap(x_{1}^{3},x_{3},x_{5})\cap(x_{2},x_{3},x_{4}^{2})\cap(x_{2},x_{3}^{5},x_{5})\cap(x_{2},x_{4},x_{5}^{2})\cap(x_{1}^{3},x_{2}^{4},x_{3}^{5},x_{5})\cap(x_{1}^{3},x_{2}^{4},x_{4},x_{5}^{2})\cap(x_{2},x_{3}^{5},x_{4}^{2},x_{5}^{2})\cap(x_{1}^{3},x_{2}^{4},x_{3}^{5},x_{4}^{2},x_{5}^{2})$.
\end{center}

\end{exmp}

\begin{exmp}\label{example2}
Let $D$ be the following oriented weighted graph 
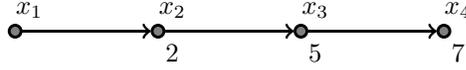
\begin{center}
\begin{tikzpicture}[line width=1.1pt,scale=0.95]
		\tikzstyle{every node}=[inner sep=0pt, minimum width=4.5pt] 

\draw [->] (0,0) --(1.9,0); %x1 a x2
\draw [->] (2.1,0) --(3.9,0); %x2 a x3
\draw [->] (4.1,0) --(5.9,0); %x3 a x4

	\draw (0,0) node (x1) [draw, circle, fill=gray] {};
	\draw (2,0) node (x2) [draw, circle, fill=gray] {};
		\draw (4,0) node (x3) [draw, circle, fill=gray] {};
			\draw (6,0) node (x4) [draw, circle, fill=gray] {};

\node at (.2,.3) {$x_{1}$};
\node at (2.2,.3) {$x_{2}$};
\node at (4.2,.3) {$x_{3}$};
\node at (6.2,.3) {$x_{4}$};

\node at (2.2,-.3) {$2$};
\node at (4.2,-.3) {$5$};
\node at (6.2,-.3) {$7$};

\end{tikzpicture}
\end{center} 
\noindent
Hence, $I(D)=(x_{1}x_{2}^{2},x_{2}x_{3}^{5},x_{3}x_{4}^{7})$. By Theorem \ref{thm-irredDecom} and Theorem \ref{decomposition}, the irreducible decomposition of $I(D)$ is: 
\begin{center}
$I(D)=(x_{1},x_{3})\cap(x_{2}^{2},x_{3})\cap(x_{2},x_{4}^{7})\cap(x_{1},x_{3}^{5},x_{4}^{7})\cap(x_{2}^{2},x_{3}^{5},x_{4}^{7})$.
\end{center}

\end{exmp}

\vspace{1ex}

In Example \ref{example1} and Example \ref{example2}, $I(D)$ has embedding primes. Furthermore the monomial ideal $(V(D))$ is an associated prime of $I(D)$ in Example \ref{example1}. Proposition \ref{prop-unicycle} and Remark \ref{ass-primes} give a combinatorial criterion for to decide when $(V(D))\in {\rm Ass}(I(D))$.

\section{Unmixed weighted oriented graphs}
Let $D=(V,E,w)$ be a weighted oriented graph whose underlying graph is $G=(V,E)$. In this section we characterize the unmixed property of $I(D)$ and we prove that this property is closed under c-minors. In particular if $G$ is a bipartite graph or $G$ is a whisker or $G$ is  a cycle, we give an effective (combinatorial) characterization of this property.  

\begin{defn}
An ideal $I$ is {\it unmixed\/} if each one of its associated primes has the same height.
\end{defn}

\begin{thm}\label{unmixed}
The following conditions are equivalent: 
\begin{enumerate}[1)]
\item $I(D)$ is unmixed.
\item Each strong vertex cover of $D$ has the same cardinality. 
\item $I(G)$ is unmixed and $L_{3}(C)=\emptyset$ for each strong vertex cover $C$ of $D$.
\end{enumerate}
\end{thm}

\proof Let $C_{1},\ldots , C_{\ell}$ be the strong vertex covers of $D$. By Remark \ref{ass-primes}, the associated primes of $I(D)$ are $P_{1},\ldots, P_{\ell}$, where $P_{i}=rad(I_{C_{i}})=(C_{i})$ for $1\leq i \leq \ell$. 

\vspace{1ex}
\noindent
$1) \Rightarrow 2)$ Since $I(D)$ is unmixed, $|C_{i}|= ht(P_{i})=ht(P_{j})=|C_{j}|$ for $1\leq i<j\leq\ell$. 

\vspace{1ex}
\noindent
$2) \Rightarrow 3)$ If $C$ is a minimal vertex cover, then by Corollary \ref{C-minimal-strong}, $C\in\{C_{1},\ldots, C_{\ell}\}$. By hypothesis, $|C_{i}|=|C_{j}|$ for each $1\leq i\leq j\leq \ell$, then $C_{i}$ is a minimal vertex cover of $D$. Thus, by Lemma \ref{minimalcover-L3}, $L_{3}(C_{i})=\emptyset$. Furthermore $I(G)$ is unmixed, since $C_{1},\ldots,C_{\ell}$ are the minimal vertex covers of $G$. 

\vspace{1ex}
\noindent
$3) \Rightarrow 1)$  
By Proposition \ref{minimalcover-L3}, $C_{i}$ is a minimal vertex cover, since $L_{3}(C_{i})=\emptyset$ for each $1\leq i \leq \ell$. This implies, $C_{1},\ldots,C_{\ell}$ are the minimal vertex covers of $G$. Since $G$ is unmixed, we have $|C_{i}|=|C_{j}|$ for $1\leq i<j\leq \ell$. Therefore $I(D)$ is unmixed. 
\qed

\begin{defn}\label{star-property}
A weighted oriented graph $D$ has the {\it minimal-strong property\/} if each strong vertex cover is a minimal vertex cover. 
\end{defn}

\begin{rem}\label{star-L3}
Using Proposition \ref{minimalcover-L3}, we have that $D$ has the minimal-strong property if and only if $L_{3}(C)=\emptyset$ for each strong vertex cover $C$ of $D$. 
\end{rem}

\begin{defn}
$D^\prime$ is a {\it c-minor\/} of $D$ if there is a stable set $S$ of $D$, such that $D^\prime =D\setminus N_{G}[S]$.
\end{defn}

\begin{lem}\label{star-lemma}
If $D$ has the minimal-strong property, then $D^\prime =D\setminus N_{G}[x]$ has the minimal-strong property, for each $x\in V$.
\end{lem}

\proof
We take a strong vertex cover $C^\prime$ of $D^\prime =D\setminus N_{G}[x]$ where $x\in V$. Thus, $C=C^\prime \cup N_{D}(x)$ is a vertex cover of $D$. If $y^\prime\in L_{3}(C^\prime)$, then by Proposition \ref{remark-L3}, $N_{D^\prime}(y^\prime)\subseteq C^\prime$. Consequently, $N_{D}(y^\prime)\subseteq C^\prime \cup N_{D}(x)=C$ implying $y^\prime\in L_{3}(C)$. Hence, $L_{3}(C^{\prime})\subseteq L_{3}(C)$. Now, we take $y\in L_{3}(C)$, then $N_{D}(y)\subseteq C$. This implies $y\notin N_{D}(x)$, since $x\notin C$. Then, $y\in C^\prime$ and $N_{D^{\prime}}(y)\cup(N_{D}(y)\cap N_{D}(x))=N_{D}(y)\subseteq C=C^{\prime}\cup N_{D}(x)$. So, $N_{D^{\prime}}(y)\subseteq C^{\prime}$ implies $y\in L_{3}(C^{\prime})$. Therefore $L_{3}(C)=L_{3}(C^{\prime})$.

\noindent
Now, if $y\in L_{3}(C)=L_{3}(C^{\prime})$, then there is $z\in C^\prime\setminus L_{1}(C^{\prime})$ with $w(z)\neq 1$, such that $(z,y)\in E(D^{\prime})$. If $z\in L_{1}(C)$, then there exist $z^\prime\notin C$ such that $(z,z^\prime)\in E(D)$. Since $z^\prime\notin C$, we have $z^{\prime}\notin C^{\prime}$, then $z\in L_{1}(C^\prime)$. A contradiction, consequently $z\notin L_{1}(C)$. Hence, $C$ is strong. This implies $L_{3}(C)=\emptyset$, since $D$ has the minimal-strong property. Thus, $L_{3}(C^\prime)=L_{3}(C)=\emptyset$. Therefore $D^\prime$ has the minimal-strong property.
\qed

\begin{prop}\label{N-unmixed}
If $D$ is unmixed and $x\in V$, then $D^{\prime}=D\setminus N_{G}[x]$ is unmixed.
\end{prop}
\proof
By Theorem \ref{unmixed}, $G$ is unmixed and $D$ has the minimal-strong property. Hence, by \cite{monalg-2}, $G^\prime=G\setminus N_{G}[x]$ is unmixed. Also, by Lemma \ref{star-lemma} we have that $D^\prime$ has the minimal-strong property. Therefore, by Theorem \ref{unmixed}, $D^{\prime}$ is unmixed.
\qed

\begin{thm}\label{c-minor-umixed}
If $D$ is unmixed, then a c-minor of $D$ is unmixed.
\end{thm}
\proof
If $D^{\prime}$ is a c-minor of $D$, then there is a stable $S=\{a_{1},\ldots,a_{s}\}$ such that $D^\prime=D\setminus N_{G}[S]$. Since $S$ is stable, $D^{\prime}=(\cdots ((D\setminus N_{G}[a_{1}])\setminus N_{G}[a_{2}])\setminus \cdots )\setminus N_{G}[a_{s}]$. Hence, by induction and Proposition \ref{N-unmixed}, $D^\prime$ is unmixed. 
\qed

\begin{prop}\label{prop-mixed}
If $V(D)$ is a strong vertex cover of $D$, then $I(D)$ is mixed.
\end{prop}

\proof
By Proposition \ref{remark-L3} $V(D)$ is not minimal, since $L_{3}(V(D))=V(D)$. Therefore, by Theorem \ref{unmixed}, $I(D)$ is mixed.
\qed

\begin{rem}\label{remark-mixed}
If $V=V^{+}$, then $I(D)$ is mixed.
\end{rem}
\proof
If $x_{i}\in V$, then by Remark \ref{rem3.3} $N^{-}_{D}(x_{i})\neq \emptyset$, since $V=V^{+}$. Thus, there is $x_{j}\in V$ such that $(x_j,x_i)\in E(D)$. Also, $w(x_j)\neq 1$ and $x_{j}\in V=L_{3}(V)$. So, $V$ is a strong vertex cover. Hence, by Proposition \ref{prop-mixed}, $I(D)$ is mixed. 
\qed

\vspace{1ex}
In the following three results we assume that $D_{1},\ldots, D_{r}$ are the connected components of $D$. Furthermore $G_{i}$ is the underlying graph of $D_{i}$.

\begin{lem}\label{L3-ci}
Let $C$ be a vertex cover of $D$, then  
$L_{1}(C)=\bigcup_{i=1}^{r}L_{1}(C_{i})$ and $L_{3}(C)=\bigcup_{i=1}^{r}L_{3}(C_{i})$, where $C_{i}=C\cap V(D_{i})$.
\end{lem}

\proof
We take $x\in C$, then $x\in C_{j}$ for some $1\leq j\leq r$. Thus, $N_{D}(x)=N_{D_{j}}(x)$. In particular $N_{D}^{+}(x)=N_{D_{j}}^{+}(x)$, so $C\cap N_{D}^{+}(x)=C_{j}\cap N_{D_{j}}^{+}(x)$. Hence, $L_{1}(C)=\bigcup_{i=1}^{r}L_{1}(C_{i})$. On the other hand, 
\begin{center}
$x\in L_{3}(C) \Leftrightarrow N_{D}(x)\subseteq C\Leftrightarrow N_{D_{j}}(x)\subseteq C_{j} \Leftrightarrow x\in L_{3}(C_{j})$.
\end{center}
Therefore, $L_{3}(C)=\bigcup_{i=1}^{r}L_{3}(C_{i})$.
\qed

\begin{lem}\label{ci-strong}
Let $C$ be a vertex cover of $D$, then $C$ is strong if and only if each $C_{i}=C\cap V(D_{i})$ is strong with $i\in\{1,\ldots, r\}$.
\end{lem}

\proof
$\Rightarrow)$ We take $x\in L_{3}(C_{j})$. By Lemma \ref{L3-ci}, $x\in L_{3}(C)$ and there is $z\in N_{D}^{-}(x)\cap V^{+}$ with $z\in C\setminus L_{1}(C)$, since $C$ is strong. So, $z\in N_{D_{j}}^{-}(x)$ and $z\in V(D_{j})$, since $x\in D_{j}$. Consequently, by Lemma \ref{L3-ci}, $z\in C_{j}\setminus L_{1}(C_{j})$. Therefore $C_{j}$ is strong.

\vspace{1ex}
 \noindent
 $\Leftarrow)$ We take $x\in L_{3}(C)$, then $x\in C_{i}$ for some $1\leq i\leq r$. Then, by Lemma \ref{L3-ci}, $x\in L_{3}(C_{i})$. Thus, there is $a\in N_{D_{i}}^{-}(x)$ such that $w(a)\neq 1$ and $a\in C_{i}\setminus L_{1}(C_{i})$, since $C_{i}$ is strong. Hence, by Lemma \ref{L3-ci}, $a\in C\setminus L_{1}(C)$. Therefore $C$ is strong.
 
\qed

\begin{cor}
$I(D)$ is unmixed if and only if $I(D_{i})$ is unmixed for each $1\leq i\leq r$. 
\end{cor}

\proof
 $\Rightarrow)$ By Theorem \ref{c-minor-umixed}, since $D_{i}$ is a c-minor of $D$.
  
 \vspace{1ex}
 \noindent
 $\Leftarrow)$ By Theorem \ref{unmixed}, $G_{i}$ is unmixed thus $G$ is unmixed. Now, if $C$ is a strong vertex cover, then by Lemma \ref{L3-ci}, $C_{i}=C\cap V(D_{i})$ is a strong vertex cover. Consequently, $L_{3}(C_{i})=\emptyset$, since $I(D_{i})$ is unmixed. Hence, by Lemma \ref{L3-ci}, $L_{3}(C)=\bigcup_{i=1}^{r}L_{3}(C_{i})=\emptyset$. Therefore, by Theorem \ref{unmixed}, $I(D)$ is unmixed.
\qed

\begin{defn}
Let $G$ be a simple graph whose vertex set is $V(G)=\{x_1,\ldots, x_{n}\}$ and edge set $E(G)$. A {\it whisker\/} of $G$ is a graph $H$ whose vertex set is $V(H)=V(G)\cup \{y_{1},\ldots, y_{n}\}$ and whose edge set is $E(H)=E(G)\cup \{\{x_1,y_1\},\ldots, \{x_{n},y_{n}\}\}$.
\end{defn}

\begin{defn}
Let $D$ and $H$ be weighted oriented graphs. $H$ is a {\it weighted oriented whisker\/} of $D$ if $D\subseteq H$ and the underlying graph of $H$ is a whisker of the underlying graph of $D$.
\end{defn}

\begin{thm}\label{unmixed-whiskers}
Let $H$ a weighted oriented whisker of $D$, where $V(D)=\{x_{1},\ldots, x_{n}\}$ and $V(H)=V(D)\cup \{y_{1},\ldots, y_{n}\}$, then the following conditions are equivalents:
\begin{enumerate}[1)]
\item $I(H)$ is unmixed.
\item If $(x_{i},y_{i})\in E(H)$ for some $1\leq i\leq n$, then $w(x_{i})=1$. 
\end{enumerate}
\end{thm}

\proof
$2) \Rightarrow 1)$ We take a strong vertex cover $C$ of $H$. Suppose $x_{j}, y_{j}\in C$, then $y_{j}\in L_{3}(C)$, since $N_{D}(y_{j})=\{x_{j}\}\subseteq C$. Consequently, $(x_{j},y_{j})\in E(G)$ and $w(x_{j})\neq 1$, since $C$ is strong. This is a contradiction by condition $2)$. This implies, $|C\cap \{x_{i},y_{i}\}|=1$ for each $1\leq i\leq n$. So, $|C|=n$. Therefore, by Theorem \ref{unmixed}, $I(H)$ is unmixed.

\vspace{1ex}
\noindent
$1) \Rightarrow 2)$ By contradiction suppose $(x_{i},y_{i})\in E(H)$ and $w(x_{i})\neq 1$ for some $i$. Since $w(x_{i})\neq 1$ and by Remark \ref{rem3.3}, we have that $x_{i}$ is not a source. Thus, there is $x_{j}\in V(D)$, such that $(x_{j},x_{i})\in E(H)$. We take the vertex cover $C=\{V(D)\setminus x_{j}\}\cup \{y_{j},y_{i}\}$, then by Proposition \ref{remark-L3}, $L_{3}(C)=\{y_{i}\}$. Furthermore $N_{D}(x_{i})\setminus C=\{x_{j}\}$ and $(x_{j},x_{i})\in E(H)$, then $x_{i}\in L_{2}(C)$. Hence C is strong, since $L_{3}(C)=\{y_{i}\}$, $(x_{i},y_{i})\in E(G)$ and $w(x_{i})\neq 1$. A contradiction by Theorem \ref{unmixed}, since $I(H)$ is unmixed.
\qed

\begin{thm}\label{unmixed-bipartite}
Let $D$ be a bipartite weighted oriented graph, then $I(D)$ is unmixed if and only if 
\begin{enumerate}[1)]
\item $G$ has a perfect matching $\{\{x_{1}^{1},x_{1}^{2}\},\ldots,\{x_{s}^{1},x_{s}^{2}\}\}$  where $\{x_{1}^{1},\ldots, x_{s}^{1}\}$ and $\{x_{1}^{2},\ldots, x_{s}^{2}\}$ are stable sets. Furthermore if $\{x_{j}^{1},x_{i}^{2}\}$, $\{x_{i}^{1},x_{k}^{2}\}\in E(G)$ then $\{x_{j}^{1},x_{k}^{2}\}\in E(G)$.
%\item If $(x_{i}^{k},x_{i}^{k^\prime})\in E(D)$ where $\{k,k^{\prime}\}=\{1,2\}$, then $w(x_{i}^{k})=1$.
\item If $w(x_{j}^{k})\neq 1$ and $N_{D}^{+}(x_{j}^{k})=\{x_{i_{1}}^{k^\prime},\ldots, x_{i_{r}}^{k^\prime}\}$ where $\{k,k^\prime\}=\{1,2\}$, then $N_{D}(x_{i_{\ell}}^{k})\subseteq N_{D}^{+}(x_{j}^{k})$ and $N_{D}^{-}(x_{i_{\ell}}^{k})\cap V^{+}=\emptyset$ for each $1\leq \ell \leq r$. \end{enumerate}
\end{thm}

\proof
$\Leftarrow)$ By $1)$ and \cite[Theorem 2.5.7]{gitler-Vill}, $G$ is unmixed. We take a strong vertex cover $C$ of $D$. Suppose $L_{3}(C)\neq \emptyset$, thus there exist $x_{i}^{k}\in L_{3}(C)$. Since $C$ is strong, there is $x_{j}^{k^{\prime}}\in V^{+}$ such that $(x_{j}^{k^\prime},x_{i}^{k})\in E(D)$, $x_{j}^{k^\prime}\in C\setminus L_{1}(C)$ and $\{k, k^\prime\}=\{1,2\}$. Furthermore $N_{D}^{+}(x_{j}^{k^\prime})\subseteq C$, since $x_{j}^{k^{\prime}}\notin L_{1}(C)$. Consequently, by $3)$, $N_{D}(x_{i}^{k^\prime})\subseteq N_{D}^{+}(x_{j}^{k^\prime})\subseteq C$ and $N_{D}^{-}(x_{i}^{k^{\prime}})\cap V^{+}=\emptyset$. A contradiction, since $x_{i}^{k^{\prime}}\in L_{3}(C)$ and $C$ is strong. Hence $L_{3}(C)=\emptyset$ and $D$ has the strong-minimal property. Therefore $I(D)$ is unmixed, by Theorem \ref{unmixed}.

\vspace{1ex}
\noindent
$\Rightarrow)$ By Theorem \ref{unmixed}, $G$ is unmixed. Hence, by \cite[Theorem 2.5.7]{gitler-Vill}, $G$ satisfies $1)$.

\vspace{1ex}
\noindent
If $w(x_{j}^{k})\neq 1$, then we take $C=N_{D}^{+}(x_{j}^{k})\cup\{x_{i}^{k}\mid N_{D}(x_{i}^{k})\not\subseteq N_{D}^{+}(x_{j}^{k})\}$ and $k^{\prime}$ such that $\{k,k^{\prime}\}=\{1,2\}$. If $\{x_{i}^{k},x_{i^{\prime}}^{k^{\prime}}\}\in E(G)$ and $x_{i}^{k}\notin C$, then $x_{i^{\prime}}^{k^{\prime}}\in N_{D}(x_{i}^{k})\subseteq N_{D}^{+}(x_{j}^{k})\subseteq C$. This implies, $C$ is a vertex cover of $D$. Now, if $x_{i_{1}}^{k}\in L_{3}(C)$, then $N_{D}(x_{i_{1}}^{k})\subseteq C$. Consequently  $N_{D}(x_{i_{1}}^{k})\subseteq N_{D}^{+}(x_{j}^{k})$ implies $x_{i_{1}}^{k}\notin C$. A contradiction, then $L_{3}(C)\subseteq N_{D}^{+}(x_{j}^{k})$. Also, $N_{G}^{-}(x_{j}^{k})\neq \emptyset$, since $w(x_{j}^{k})\neq 1$. Thus $x_{j}^{k}\in L_{2}(C)$, since $N_{G}^{-}(x_{j}^{k})\cap C=\emptyset$. Hence $C$ is strong, since $L_{3}(C)\subseteq N_{D}^{+}(x_{j}^{k})$ and $x_{j}^{k}\in V^{+}$. Furthermore $\{x_{1}^{\prime},\ldots, x_{s}^{\prime}\}$ is a minimal vertex cover, then by Theorem \ref{unmixed} $|C|=s$, since $D$ is unmixed. We assume $N_{D}^{+}(x_{j}^{k})=\{x_{i_{1}}^{k^{\prime}},\ldots x_{i_{r}}^{k^{\prime}}\}$. Since $C$ is minimal, $x_{i_{\ell}}^{k}\notin C$ for each $1\leq \ell \leq r$. So, $N_{D}(x_{i_{\ell}}^{k})\subseteq N_{D}^{+}(x_{j}^{k})$. Now, suppose $z\in N_{D}^{-}(x_{i_{\ell}}^{k})\cap V^{+}$, then $z=x_{i_{\ell^{\prime}}}^{k^{\prime}}$ for some $1\leq\ell^{\prime}\leq r$, since $N_{D}(x_{i_{\ell}}^{k})\subseteq N_{D}^{+}(x_{j}^{k})$. We take $C^\prime=N_{D}^{+}(x_{j}^{k})\cup\{x_{i}^{k}\mid i\notin\{i_{1},\ldots,i_{r}\}\}\cup N_{D}^{+}(x_{i_{\ell^{\prime}}}^{k^\prime})$. Since $N_{D}(x_{i_{u}}^{k})\subseteq N_{D}^{+}(x_{j}^{k})$ for each $1\leq u\leq r$, we have that $C^{\prime}$ is a vertex cover. If $\{x_{q}^{k},x_{q}^{k^\prime}\}\cap L_{3}(C)=\emptyset$, then $\{x_{q}^{k},x_{q}^{k^\prime}\}\subseteq C^\prime$, so $x_{q}^{k^{\prime}}\in N_{D}^{+}(x_{j}^{k})$ implies $q\in \{i_{1},\ldots, i_{r}\}$. Consequently, $x_{q}^{k}\in N_{D}^{+}(x_{i_{\ell^{\prime}}}^{k^{\prime}})$, since $x_{q}^{k}\in C^{\prime}$. This implies, $(x_{j}^{k},x_{q}^{k^\prime}), (x_{i_{\ell^{\prime}}}^{k^\prime}, x_{q}^{k})\in E(D)$. Moreover, $N_{D}^{+}(x_{i_{\ell^{\prime}}}^{k^\prime})\cup N_{D}^{+}(x_{j}^{k})\subseteq C^\prime$, then $x_{i_{\ell^{\prime}}}^{k^\prime}\notin L_{1}(C^\prime)$ and $x_{j}^{k}\notin L_{1}(C^{\prime})$. Thus, $C^\prime$ is strong, since $x_{j}^{k},x_{i_{\ell^{\prime}}}^{k^{\prime}}\in V^{+}$. Furthermore, by Theorem \ref{unmixed}, $|C^{\prime}|=s$. But $x_{i_{\ell}}^{k^{\prime}}\in N_{D}^{+}(x_{j}^{k})$ and $x_{i_{\ell}}^{k}\in N_{D}^{+}(x_{i_{\ell^{\prime}}}^{k^{\prime}})$, hence $x_{i_{\ell}}^{k^{\prime}}, x_{i_{\ell}}^{k}\in C^{\prime}$. A contradiction, so $N_{D}^{-}(x_{i_{\ell}}^{k})\cap V^{+}=\emptyset$. Therefore $D$ satisfies $2)$. 
\qed

\begin{lem}\label{lem1-cycle}
If the vertices of $V^{+}$ are sinks, then $D$ has the minimal-strong property.
\end{lem}

\proof
We take a strong vertex cover $C$ of $D$. Hence, if $y\in L_{3}(C)$, then there is $(z,y)\in E(D)$ with $z\in V^{+}$. Consequently, by hypothesis, $z$ is a sink. A contradiction, since $(z,y)\in E(D)$. Therefore, $L_{3}(C)=\emptyset$ and $C$ is a minimal vertex cover.   
\qed

\begin{lem}\label{lem2-cycle}
Let $D$ be a weighted oriented graph, where $G\simeq C_{n}$ with $n\geq 6$. Hence, $D$ has the minimal-strong property if and only if the vertices of $V^{+}$ are sinks. 
\end{lem}

\proof
$\Leftarrow)$ By Lemma \ref{lem1-cycle}. 

\vspace{1ex}
\noindent
$\Rightarrow)$ By contradiction, suppose there is $(z,y)\in E(D)$, with $z\in V^{+}$. We can assume $G=(x_{1},x_{2},\ldots,x_{n},x_{1})\simeq C_{n}$, with $x_{2}=y$ and $x_{3}=z$. We take a strong vertex cover $C$ in the following form: $C=\{x_{1},x_{3},\ldots, x_{n-1}\}\cup \{x_{2}\}$ if $n$ is even or $C=\{x_{1},x_{3},\ldots, x_{n-2}\}\cup \{x_{2},x_{n-1}\}$ if $n$ is odd. Consequently, if $x\in C$ and $N_{D}(x)\subseteq C$, then $x=x_{2}$. Hence, $L_{3}(C)=\{x_{2}\}$. Furthermore $(x_{3},x_{2})\in E(D)$ with $x_{3}\in V^{+}$. Thus, $x_{3}$ is not a source, so, $(x_{4},x_{3})\in E(D)$. Then, $x_{3}\in L_{2}(C)$. This implies $C$ is a strong vertex cover. But $L_{3}(C)\neq \emptyset$. A contradiction, since $D$ has the minimal-strong property.
\qed

%AQUI VA LA FIGURA 
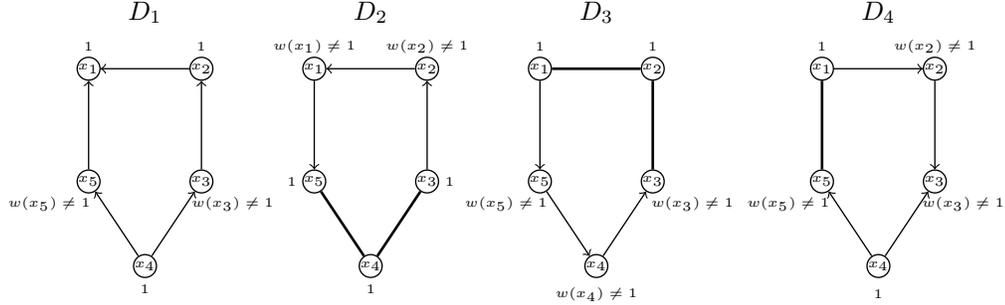
\begin{center}

\begin{figure}[h]
\begin{tikzpicture}[line width=.5pt,scale=0.75]
		\tikzstyle{every node}=[inner sep=0pt, minimum width=5.5pt] 

\tiny{
%D1
\node[draw, circle] (1) at (0,0){$x_{1}$};
\node[draw, circle] (2) at (2,0) {$x_{2}$};
\node[draw, circle] (5) at (0,-2) {$x_{5}$};
\node[draw, circle] (3) at (2,-2){$x_{3}$};
\node[draw, circle] (4) at (1,-3.5){$x_{4}$};

\draw[->] (2) to (1);
\draw[->] (5) to (1);
\draw[->] (4) to (5);
\draw[->] (4) to (3);
\draw[->] (3) to (2);

\node at (0,.4) {\tiny{$1$}};%x1
\node at (2,.4) {\tiny{$1$}};%x2
\node at (-.7,-2.4) {\tiny{$w(x_{5})\neq 1$}};%x5
\node at (1,-3.9) {\tiny{$1$}};%x4
\node at (2.55,-2.4) {\tiny{$w(x_{3})\neq 1$}};%x3
\node at (1,1) {\normalsize{$D_{1}$}};

%%%%%%%%%%
%D2
\node[draw, circle] (1) at (4,0){$x_{1}$};
\node[draw, circle] (2) at (6,0) {$x_{2}$};
\node[draw, circle] (5) at (4,-2) {$x_{5}$};
\node[draw, circle] (3) at (6,-2){$x_{3}$};
\node[draw, circle] (4) at (5,-3.5){$x_{4}$};

\draw[->] (1) to (5);
\draw[line width=1pt] (5) -- (4);
\draw[->] (2) to (1);
\draw[line width=1pt] (4) -- (3);
\draw[->] (3) to (2);

\node at (4,.4) {\tiny{$w(x_{1})\neq 1$}};%x1
\node at (6,.4) {\tiny{$w(x_{2})\neq 1$}};%x2
\node at (3.6,-2) {\tiny{$1$}};%x5
\node at (5,-3.9) {\tiny{$1$}};%x4
\node at (6.4,-2) {\tiny{$1$}};%x3
\node at (5,1) {\normalsize{$D_{2}$}};

%%%%%%%%%%%%
%D3
\node[draw, circle] (1) at (8,0){$x_{1}$};
\node[draw, circle] (2) at (10,0) {$x_{2}$};
\node[draw, circle] (5) at (8,-2) {$x_{5}$};
\node[draw, circle] (3) at (10,-2){$x_{3}$};
\node[draw, circle] (4) at (9,-3.5){$x_{4}$};

\draw[line width=1pt] (2) -- (1);
\draw[->] (1) to (5);
\draw[->] (5) to (4);
\draw[->] (4) to (3);
\draw[line width=1pt] (3) -- (2);

\node at (8,0.4) {\tiny{$1$}};%x1
\node at (10,.4) {\tiny{$1$}};%x2
\node at (7.4,-2.4) {\tiny{$w(x_{5})\neq 1$}};%x5
\node at (9,-4) {\tiny{$w(x_{4})\neq 1$}};%x4
\node at (10.7,-2.4) {\tiny{$w(x_{3})\neq 1$}};%x3
\node at (9,1) {\normalsize{$D_{3}$}};

%%%%%%%%
%D4
\node[draw, circle] (1) at (13,0){$x_{1}$};
\node[draw, circle] (2) at (15,0) {$x_{2}$};
\node[draw, circle] (5) at (13,-2) {$x_{5}$};
\node[draw, circle] (3) at (15,-2){$x_{3}$};
\node[draw, circle] (4) at (14,-3.5){$x_{4}$};

\draw[->] (1) to (2);
\draw[->] (2) to (3);
\draw[->] (4) to (3);
\draw[->] (4) to (5);
\draw[line width=1pt] (5) -- (1);

\node at (13,0.4) {\tiny{$1$}};%x1
\node at (15,.4) {\tiny{$w(x_{2})\neq 1$}};%x2
\node at (12.4,-2.4) {\tiny{$w(x_{5})\neq 1$}};%x5
\node at (14,-4) {\tiny{$1$}};%x4
\node at (15.5,-2.4) {\tiny{$w(x_3)\neq 1$}};%x3
\node at (14,1) {\normalsize{$D_{4}$}};
}
\end{tikzpicture}

% \caption{......} 
\label{fig:1}
\end{figure}

\end{center}

\begin{thm}
If $G\simeq C_{n}$, then $I(D)$ is unmixed if and only if one of the following conditions hold:
\begin{enumerate}[1)]
\item $n=3$ and there is $x\in V(D)$ such that $w(x)=1$.
\item $n\in\{4,5,7\}$ and the vertices of $V^{+}$ are sinks.
\item $n=5$, there is $(x,y)\in E(D)$ with $w(x)=w(y)=1$ and $D\not\simeq D_{1}, D\not\simeq D_{2}, D\not\simeq D_{3}$. 
\item $D\simeq D_{4}$.
\end{enumerate}
\end{thm}

\proof
$\Rightarrow)$ By Theorem \ref{unmixed}, $D$ has the minimal- strong property and $G$ is unmixed. Then, by \cite[Exercise 2.4.22]{gitler-Vill}, $n\in\{3,4,5,7\}$. If $n=3$, then by Remark \ref{remark-mixed}, $D$ satisfies $1)$. If $n=7$, then by Lemma \ref{lem2-cycle}, $D$ satisfies $2)$. Now suppose $n=4$ and $D$ does not satisfies $2)$, then we can assume $x_{1}\in V^{+}$ and $(x_{1},x_{2})\in E(D)$. Consequently, $(x_{4},x_{1})\in E(G)$, since $w(x_{1})\neq 1$. Furthermore, $\mathcal{C}=\{x_1,x_2,x_3\}$ is a vertex cover with $L_{3}(\mathcal{C})=\{x_2\}$. Thus, $x_{1}\in L_{2}(\mathcal{C})$ and $(x_1,x_2)\in E(D)$ so $\mathcal{C}$ is strong. A contradiction, since $\mathcal{C}$ is not minimal. This implies $D$ satisfies $2)$. Finally suppose $n=5$. If $D\simeq D_{1}$, then $\mathcal{C}_{1}=\{x_1,x_2,x_3,x_5\}$ is a vertex cover with $L_{3}(\mathcal{C}_{1})=\{x_1,x_2\}$. Also $(x_{5},x_1), (x_{3},x_{2})\in E(D)$ with $x_{5},x_{3}\in V^{+}$. Consequently, $\mathcal{C}_{1}$ is strong, since $x_{5},x_{3}\in L_{2}(\mathcal{C}_{1})$. A contradiction, since $\mathcal{C}_{1}$ is not minimal. If $D\simeq D_{2}$, then $\mathcal{C}_{2}=\{x_1,x_2,x_4,x_5\}$ is a vertex cover where $L_{3}(\mathcal{C}_{2})=\{x_1,x_5\}$ and $(x_{2},x_1), (x_{1},x_{5})\in E(D)$ with $x_{2}, x_{1}\in V^{+}$. Hence, $\mathcal{C}_{2}$ is strong, since $x_{2},x_{1}\notin L_{1}(\mathcal{C}_{2})$. A contradiction, since $\mathcal{C}_{2}$ is not minimal. If $D\simeq D_{3}$, $\mathcal{C}_{3}=\{x_2,x_3,x_4,x_5\}$ is a vertex cover where $L_{3}(\mathcal{C}_{3})=\{x_3,x_4\}$ and $(x_{4},x_3), (x_{5},x_{4})\in E(D)$ with $x_{4},x_{5}\in V^{+}$. Thus, $\mathcal{C}_{3}$ is strong, since $x_{4},x_{5}\notin L_{1}(\mathcal{C}_{3})$. A contradiction, since $\mathcal{C}_{3}$ is not minimal. Now, since $n=5$ and by $3)$ we can assume $(x_2,x_3)\in E(D)$, $x_{2},x_{3}\in V^{+}$ and there are not two adjacent vertices with weight $1$. Since $x_2\in V^{+}$, $(x_1,x_2)\in E(D)$. Suppose there are not $3$ vertices $z_1,z_2,z_3$ in $V^{+}$ such that $(z_1,z_2,z_3)$ is a path in $G$, then $w(x_4)=w(x_1)=1$. Furthermore, $w(x_5)\neq 1$, since there are not adjacent vertices with weight $1$. So, $\mathcal{C}_4=\{x_2,x_3,x_4,x_5\}$ is a vertex cover of $D$, where $L_{3}(\mathcal{C}_4)=\{x_3,x_4\}$. Also $(x_2,x_3)\in E(G)$ with $w(x_{2})\neq 1$. Hence, if $(x_{3},x_4)\in E(D)$ or $(x_{5},x_{4})\in E(D)$, then $\mathcal{C}_4$ is strong, since $x_{3},x_{5}\in V^{+}$. But $\mathcal{C}_{4}$ is not minimal. Consequently, $(x_{4},x_3),(x_4,x_5)\in E(D)$ and $D\simeq D_{4}$. Now, we can assume there is a path $(z_{1},z_{2},z_{3})$ in $D$ such that $z_{1},z_{2},z_{3}\in V^{+}$. Since there are not adjacent vertices with weight $1$, we can suppose there is $z_{4}\in V^{+}$ such that $\mathcal{L}=(z_1,z_2,z_3,z_4)$ is a path. We take $\{z_5\}=V(D)\setminus V(\mathcal(L))$ and we can assume $(z_2,z_3)\in E(D)$. This implies, $(z_1,z_2)$, $(z_5,z_1)\in E(D)$, since $z_1,z_2\in V^{+}$. Thus, $\mathcal{C}_{5}=\{z_1,z_2,z_3,z_4\}$ is a vertex cover with $L_{3}(\mathcal{C}_{5})=\{z_2,z_3\}$. Then $\mathcal{C}_{5}$ is strong, since $(z_1,z_2), (z_2,z_3)\in E(D)$ with $z_2\in L_{3}(\mathcal{C}_{5})$ and $z_1\in L_{2}(\mathcal{C}_{5})$. A contradiction, since $\mathcal{C}_{5}$ is not minimal.  

\vspace{1ex}
\noindent
$\Leftarrow)$ If $n\in\{3,4,5,7\}$, then by \cite[Exercise 2.4.22]{gitler-Vill} $G$ is unmixed. By Theorem \ref{unmixed}, we will only prove that $D$ has the minimal-strong property. If $D$ satisfies $2)$, then by Lemma \ref{lem1-cycle}, $D$ has the minimal-strong property. If $D$ satisfies $1)$ and $\mathcal{C}$ is a strong vertex cover, then by Proposition \ref{prop-unicycle}, $|\mathcal{C}|\leq 2$. This implies $\mathcal{C}$ is minimal. Now, suppose $n=5$ and $\mathcal{C}^\prime$ is a strong vertex cover of $D$ with $|\mathcal{C}^\prime|\geq 4$. If $D\simeq D_4$, then $x_2,x_5\notin L_{3}(\mathcal{C}^\prime)$, since $(N_{D}^{-}(x_{2})\cup N_{D}^{-}(x_{5}))\cap V^{+}=\emptyset$. So $N_{D}(x_{2})\not\subseteq \mathcal{C}^{\prime}$ and $N_{D}(x_{5})\not\subseteq \mathcal{C}^{\prime}$. Consequently, $x_{1}\notin \mathcal{C}^{\prime}$ implies $\mathcal{C}^{\prime}=\{x_2,x_3,x_4,x_5\}$. But $x_4\in L_{3}(\mathcal{C}^{\prime})$ and $N_{D}^{-}(x_{4})=\emptyset$. A contradiction, since $\mathcal{C}^{\prime}$ is strong. Now assume $D$ satisfies $3)$. Suppose there is a path $\mathcal{L}=(x_1,x_2,x_3)$ in $G$ such that $w(x_1)=w(x_2)=w(x_3)=1$. We can suppose $(x_4,x_5)\in E(D)$ where $V(D)\setminus V(\mathcal{L})=\{x_4,x_5\}$. Since $w(x_1)=w(x_3)=1$, $x_2\notin L_{3}(\mathcal{C}^\prime)$. If $x_2\notin \mathcal{C}^\prime$, then $\mathcal{C}^\prime=\{x_1,x_3,x_4,x_5\}$ and  $x_4\in L_{3}(\mathcal{C}^\prime)$. But $N_{D}^{-}(x_4)=\{x_3\}$ and $w(x_3)=1$. A contradiction, hence $x_2\in \mathcal{C}^\prime$. We can assume $x_3\notin\mathcal{C}^\prime$, since $x_2\notin L_{3}(\mathcal{C}^\prime)$. This implies $\mathcal{C}^\prime=\{x_1,x_2,x_4,x_5\}$ and $L_{3}(\mathcal{C}^\prime)=\{x_1,x_5\}$. Thus, $(x_5,x_1)\in E(D)$, $x_{5},x_{4}\in V^{+}$. Consequently $(x_3,x_4)\in E(D)$, since $x_4\in V^{+}$. A contradiction, since $D\not\simeq D_{2}$. Hence, there are not three consecutive vertices whose weights are $1$. Consequently, since $D$ satisfies $3)$, we can assume $w(x_1)=w(x_2)=1$, $w(x_3)\neq 1$ and $w(x_5)\neq 1$. If $w(x_4)=1$, then $x_3,x_5\notin L_{3}(\mathcal{C}^\prime)$ since $N_{D}(x_3,x_5)\cap V^{+}=\emptyset$. This implies $N_{D}(x_{3})\not\subseteq \mathcal{C}^{\prime}$ and $N_{D}(x_5)\not\subseteq \mathcal{C}^{\prime}$. Then, $x_4\notin \mathcal{C}^\prime$ and $\mathcal{C}^\prime=\{x_1,x_2,x_3,x_5\}$. Thus, $(x_5,x_1), (x_3,x_2)\in E(D)$, since $L_{3}(\mathcal{C}^\prime)=\{x_1,x_2\}$. Consequently, $(x_4,x_5), (x_4,x_3)\in E(D)$, since $x_5,x_3\in V^{+}$. A contradiction, since $D\not \simeq D_1$. So, $w(x_4)\neq 1$ and we can assume $(x_5,x_4)\in E(D)$, since $x_4\in V^{+}$. Furthermore $(x_1,x_5)\in E(D)$, since $x_5\in V^{+}$. Hence, $(x_3,x_4)\in E(D)$, since $D\not\simeq D_{3}$. Then $(x_2,x_3)\in E(D)$, since $x_3\in V^{+}$. This implies $x_1,x_2,x_3,x_5\notin L_{3}(\mathcal{C}^{\prime})$, since $N^{-}_{D}(x_{i})\cap V^{+}=\emptyset$ for $i\in \{1,2,3,5\}$. A contradiction, since $|\mathcal{C}^{\prime}|\geq 4$. Therefore $D$ has the minimal-strong property. 
\qed

\section{Cohen-Macaulay weighted oriented graphs}
In this section we study the Cohen-Macaulayness of $I(D)$. In particular we give a combinatorial characterization of this property when $D$ is a path or $D$ is complete. Furthermore, we show the Cohen-Macaulay property depends of the characteristic of $K$.

\begin{defn}
The weighted oriented graph $D$ is {\it Cohen- Macaulay\/} over the field $K$ if the ring $R/I(D)$ is Cohen-Macaulay. 

\end{defn}

\begin{rem}\label{radical}
If $G$ is the underlying graph of $D$, then $rad(I(D))=I(G)$.
\end{rem}

\begin{prop}
If $I(D)$ is Cohen-Macaulay, then $I(G)$ is Cohen-Macaulay and $D$ has the minimal-strong property.
%$L_{3}(C)=\emptyset$ for each strong vertex cover $C$ of $D$.
\end{prop}

\begin{proof}
 By Remark \ref{radical}, $I(G)=rad(I(D))$, then by \cite[Theorem 2.6]{Radical-Herzog}, $I(G)$ is Cohen-Macaulay.  Furthermore $I(D)$ is unmixed, since $I(D)$ is Cohen-Macaulay. Hence, by Theorem \ref{unmixed}, $D$ has the minimal-strong property.
 %$L_{3}(C)=\emptyset$ for each strong vertex cover $C$ of $D$.
\end{proof}

\begin{exmp}
In Example \ref{example1} and Example \ref{example2} $I(D)$ is mixed. Hence, $I(D)$ is not Cohen-Macaulay, but $I(G)$ is Cohen-Macaulay.
\end{exmp}

\begin{conj}
$I(D)$ is Cohen- Macaulay if and only if $I(G)$ is Cohen-Macaulay and $D$ has the minimal-strong property. Equivalently $I(D)$ is Cohen-Macaulay if and only if $I(D)$ is unmixed and $I(G)$ is Cohen-Macaulay.
%$L_{3}(C)=\emptyset$ for each strong vertex cover $C$ of $D$.
\end{conj}

\begin{prop}\label{prop-path}
Let $D$ be a weighted oriented graph such that $V=\{x_{1},\ldots, x_{k}\}$ and whose underlying graph is a path $G=(x_{1},\ldots, x_{k})$. Then the following conditions are equivalent. 
\begin{enumerate}[1)]
\item $R/I(D)$ is Cohen-Macaulay.
\item $I(D)$ is unmixed.
\item $k=2$ or $k=4$. In the second case, if $(x_{2},x_{1})\in E(D)$ or $(x_{3},x_{4})\in E(D)$, then $w(x_{2})=1$  or $w(x_3)=1$ respectively.
\end{enumerate} 
\end{prop}

\proof
{\it 1)} $\Rightarrow$ {\it 2)} By \cite[Corollary 1.5.14]{gitler-Vill}.

\vspace{1ex}
\noindent
{\it 2)} $\Rightarrow$ {\it 3)} By Theorem \ref{unmixed-bipartite}, $G$ has a perfect matching, since $D$ is bipartite. Consequently $k$ is even and $\{x_1,x_2\}, \{x_3,x_4\}, \ldots, \{x_{k-1},x_k\}$ is a perfect matching. If $k\geq 6$, then by Theorem \ref{unmixed-bipartite}, we have $\{x_2,x_5\}\in E(G)$, since $\{x_2,x_3\}$ and $\{x_4,x_5\}\in E(G)$. A contradiction since $\{x_2,x_5\}\notin E(G)$. Therefore $k\in\{2,4\}$. Furthermore by Theorem \ref{unmixed-bipartite}, $w(x_{2})=1$ or $w(x_{3})=1$ when $(x_{2},x_{1})\in E(D)$ or $(x_{3},x_{4})\in E(D)$, respectively.

\vspace{1ex}
\noindent
{\it 3)} $\Rightarrow$ {\it 1)} We take $I=I(D)$. If $k=2$, then we can assume $(x_1,x_2)\in E(D)$. So, $I=(x_{1}x_{2}^{w(x_2)})=(x_{1})\cap (x_{2}^{w(x_{2})})$. Thus, by Remark \ref{ass-primes}, ${\rm Ass}(I)=\{(x_1),(x_2)\}$. This implies, ${\rm ht}(I)=1$ and $\dim(R/I)=k-1=1$. Also, ${\rm depth}(R/I)\geq 1$, since $(x_1,x_2)\notin {\rm Ass}(I)$. Hence, $R/I$ is Cohen-Macaulay. Now, if $k=4$, then ${\rm ht}(I)={\rm ht}({\rm rad}(I))={\rm ht}(I(G))=2$. Consequently, $\dim(R/I)=k-2=2$. Furthermore one of the following sets $\{x_2-x_1^{w(x_1)},x_3-x_4^{w(x_4)}\}$, $\{x_2-x_1^{w(x_1)},x_4-x_3^{w(x_3)}\}$, $\{x_1-x_2^{w(x_2)},x_4-x_3^{w(x_3)}\}$ is a regular sequence of $R/I$, then ${\rm depth}(R/I)\geq 2$. Therefore, $I$ is Cohen-Macaulay.
\qed

\begin{thm}
If $G$ is a complete graph, then the following conditions are equivalent. 
\begin{enumerate}[1)]
\item $I(D)$ is unmixed.
\item $I(D)$ is Cohen-Macaulay.
\item There are not $D_1,\ldots, D_{s}$ unicycles orientes subgraphs of $D$ such that $V(D_{1}),\ldots, V(D_s)$ is a partition of $V(D)$
\end{enumerate}

\end{thm}

\proof
 We take $I=I(D)$. Since $I(G)=rad(I)$ and $G$ is complete, ${\rm ht}(I)={\rm ht}(I(G))=n-1$. 

\vspace{1ex}
\noindent
{\it 1)} $\Rightarrow$ {\it 3)} Since ${\rm ht}(I)=n-1$ and $I$ is unmixed, $(x_1,\ldots,x_n)\notin {\rm Ass}(I)$. Thus, by Remark \ref{ass-primes}, $V(D)$ is not a strong vertex cover of $D$. Therefore, by Proposition \ref{prop-unicycle}, $D$ satisfies $3)$. 

\vspace{1ex}
\noindent
{\it 3)} $\Rightarrow$ {\it 2)} By Proposition \ref{prop-unicycle}, $V(D)$ is not a strong vertex cover of $D$. Consequently, by Remark \ref{ass-primes}, $(x_1,\ldots, x_n)\notin {\rm Ass}(I)$. This implies, ${\rm depth}(R/I)\geq 1$. Furthermore, $\dim(R/I)=1$, since ${\rm ht}(I)=n-1$. Therefore $I$ is Cohen-Macaulay. 

\vspace{1ex}
\noindent
{\it 2)} $\Rightarrow$ {\it 1)} By \cite[Corollary 1.5.14]{gitler-Vill}.

\qed

\vspace{1ex}

If $D$ is complete or $D$ is a path, then the Cohen-Macaulay property does not depend of the field $K$. It is not true in general, see the following example.

\begin{exmp}
Let $D$ be the following weighted oriented graph:
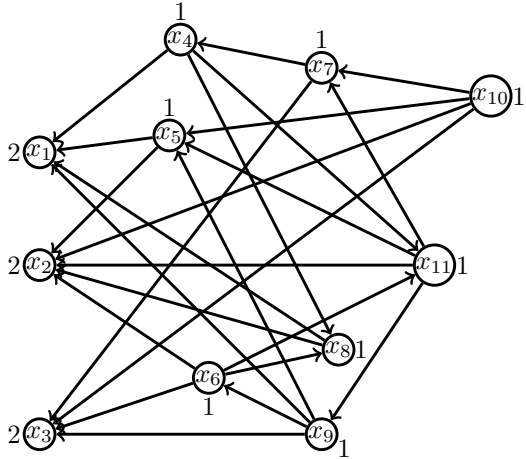
\begin{center}
\begin{tikzpicture}[line width=1.1pt,scale=0.75]
		\tikzstyle{every node}=[inner sep=0pt, minimum width=6.5pt]

\node[draw, circle] (3) at (0,0){$x_{3}$};
\node[draw, circle] (2) at (0,3) {$x_{2}$};
\node[draw, circle] (1) at (0,5) {$x_{1}$};
\node[draw, circle] (9) at (5,0){$x_{9}$};
\node[draw, circle] (8) at (5.3,1.5){$x_{8}$};
\node[draw, circle] (7) at (5,6.5) {$x_{7}$};
\node[draw, circle] (6) at (3,1) {$x_{6}$};
\node[draw, circle] (5) at (2.3,5.3) {$x_{5}$};
 \node[draw, circle] (4) at (2.5,7) {$x_{4}$};
\node[draw, circle] (11) at (7,3) {$x_{11}$};
\node[draw, circle] (10) at (8,6){$x_{10}$};

\draw[->] (4) to (1);
\draw[->] (8) to (1);
\draw[->] (5) to (1);
\draw[->] (9) to (1);

\draw[->] (10) to (2);
\draw[->] (5) to (2);
\draw[->] (11) to (2);
\draw[->] (8) to (2);
\draw[->] (6) to (2);

\draw[->] (7) to (3);
\draw[->] (10) to (3);
\draw[->] (6) to (3);
\draw[->] (9) to (3);

\draw[->] (4) to (8);
\draw[->] (7) to (4);
\draw[->] (4) to (11);

\draw[->] (10) to (5);
\draw[->] (11) to (5);
\draw[->] (9) to (5);

\draw[->] (9) to (6);
\draw[->] (6) to (8);
\draw[->] (6) to (11);

\draw[->] (10) to (7);
\draw[->] (11) to (7);

\draw[->] (11) to (9);

\node at (-.45,0) {$2$};
\node at (-.45,3) {$2$};
\node at (-.45,5) {$2$};
\node at (3,.5) {$1$};%6
\node at (2.5,7.5) {$1$};%4
\node at (2.3,5.8) {$1$};%5
\node at (5.4,-.25) {$1$};%9
\node at (5.7,1.5) {$1$};%8
\node at (5,7) {$1$};%7
\node at (8.5,6) {$1$};%10
\node at (7.5,3) {$1$};%11

\end{tikzpicture}
\end{center}

\noindent
Hence, 
\begin{center}
$I(D)=(x_{1}^{2}x_{4}, x_{1}^{2}x_{8}, x_{1}^{2}x_{5}, x_{1}^{2}x_{9}, x_{2}^{2}x_{10},x_{2}^{2}x_{5}, x_{2}^{2}x_{11},x_{2}^{2}x_{8}, x_{2}^{2}x_{6}, x_{3}^{2}x_{7}, x_{3}^{2}x_{10}, x_{3}^{2}x_{6},$

$x_{3}^{2}x_{9}, x_{4}x_{8}, x_{4}x_{7}, x_{4}x_{11}, x_{5}x_{10}, x_{5}x_{9}, x_{5}x_{11}, x_{6}x_{8}, x_{6}x_{9}, x_{6}x_{11}, x_{7}x_{10}, x_{7}x_{11}, x_{9}x_{11})$.
\end{center}
\noindent
By \cite[Example 2.3]{ex-CM-char}, $I(G)$ is Cohen- Macaulay when the characteristic of the field $K$ is zero but it is not Cohen-Macaulay in characteristic $2$. Consequently, $I(D)$ is not Cohen-Macaulay when the characteristic of $K$ is $2$. Also, $I(G)$ is unmixed. Furthermore, by Lemma \ref{lem1-cycle}, $I(D)$ has the minimal-strong property, then $I(D)$ is unmixed. Using {\it Macaulay$2$\/} \cite{mac2} we show that $I(D)$ is Cohen-Macaulay when the characteristic of $K$ is zero.
\end{exmp}

\bibliographystyle{amsplain}

\end{document}